\documentclass[12pt]{amsart}
\usepackage{pinlabel}
\usepackage[all]{xy}

\usepackage{amsmath,graphics,xypic}
\usepackage{amsfonts,amssymb}
\usepackage[all]{xy}
\theoremstyle{plain}
\newtheorem{theorem*}{Theorem}
\newtheorem*{lemma*}{Lemma}
\newtheorem{corollary*}{Corollary}
\newtheorem*{proposition*}{Proposition}
\newtheorem{conjecture*}{Conjecture}
\newtheorem{theorem}{Theorem}[section]
\newtheorem{lemma}[theorem]{Lemma}

\newtheorem{proposition}[theorem]{Proposition}

\newtheorem{conjecture}[theorem]{Conjecture}

\theoremstyle{remark}

\newtheorem*{definition}{Definition}

\theoremstyle{definition}

\def\stateh{\mathcal{H}}
\def\statefs{\mathcal{S}}

\def\mm{\mathfrak{m}}

\def\gl{\mbox{GL}} \def\Q{\Bbb{Q}}  \def\Z{\Bbb{Z}} \def\R{\Bbb{R}} \def\C{\Bbb{C}}
\def\N{\Bbb{N}}   \def\ll{\langle} \def\rr{\rangle}
 \def\a{\alpha} \def\g{\gamma}  \def\bp{\begin{pmatrix}}
\def\sm{\setminus} \def\ep{\end{pmatrix}} \def\bn{\begin{enumerate}} 
  \def\div{\mbox{div}} \def\en{\end{enumerate}}
\def\ba{\begin{array}} \def\ea{\end{array}}  
 \def\S{\Sigma}  \def\a{\alpha} \def\b{\beta} \def\ti{\tilde}
\def\id{\mbox{id}}  \def\im{\mbox{Im}} 
  \def\ker{\mbox{Ker}}
\def\ker{\mbox{Ker}}\def\be{\begin{equation}} \def\ee{\end{equation}} 
   
 \def\hom{\mbox{Hom}}  
 \def\aut{\mbox{Aut}}

          \def\znt{\Z^n[t^{\pm 1}]} 
\def\zt{\Z[t^{\pm 1}]}    
 \def\tnphi{||\phi||_T} 
     \def\fr12{\frac{1}{2}} \def\z12{\Z[\fr12]}

\def\tpm {[t^{\pm 1}]}

\def\i{\iota}

\def\G{\Gamma}

\def\G{\Gamma}
\def\PP{\mathcal{P}}

\def\deg{\mbox{deg}}

\begin{document}

\title{Twisted Alexander polynomials and fibered 3--manifolds}

\author{Stefan Friedl}
\address{University of Warwick, Coventry, UK}
\email{s.k.friedl@warwick.ac.uk}
\author{Stefano Vidussi}
\address{Department of Mathematics, University of California,
Riverside, CA 92521, USA} \email{svidussi@math.ucr.edu}
\thanks{The second author was partially supported by NSF grant \#0906281.}

\date{\today}
\begin{abstract}
In a series of papers the authors proved that twisted Alexander polynomials detect fibered 3-manifolds, and they showed that this implies that
a closed 3--manifold $N$ is fibered if and only if $S^1\times N$ is symplectic. In this note we summarize some of the key ideas of the proofs.
We also give new evidence to the conjecture that if $M$ is a symplectic 4--manifold with a free $S^1$--action, then the orbit space is fibered.
\end{abstract}
\maketitle

\section{Introduction}

\subsection{Definitions and previous results}
A \emph{manifold pair} is a pair $(N,\phi)$ where  $N\ne S^1\times D^2, N\ne S^1\times S^2$ is an orientable connected 3--manifold with toroidal or empty boundary
and $\phi\in H^1(N;\Z)=\hom(\pi_1(N),\Z)$ is non--trivial.
We say that a manifold pair
\emph{$(N,\phi)$ fibers over $S^{1}$}
 if there exists
 a fibration $p:N\to S^1$ such that the induced map $p_*:\pi_1(N)\to \pi_1(S^1)=\Z$ coincides with $\phi$.
 Given a manifold pair $(N,\phi)$ the \emph{Thurston norm} of $\phi$ is defined as
 \[
||\phi||_{T}=\min \{ \chi_-(\S)\, | \, \S \subset N \mbox{ properly embedded surface dual to }\phi\}.
\]
Here, given a compact surface $\S$ with connected components $\S_1\cup\dots \cup \S_k$, we define
$\chi_-(\S)=\sum_{i=1}^k \max\{-\chi(\S_i),0\}$.
Thurston
\cite{Th86} showed that this defines a seminorm on $H^1(N;\Z)$ which can be extended to a seminorm
on $H^1(N;\R)$.

Given a manifold pair $(N,\phi)$ and a homomorphism $\a:\pi_1(N)\to G$ to a finite group $G$ we can consider
the corresponding twisted Alexander polynomial $\Delta_{N,\phi}^\a\in \zt$. This invariant was initially introduced by
 Lin \cite{Lin01}, Wada \cite{Wa94} and Kirk--Livingston \cite{KL99}. We will recall the definition in Section \ref{sectionufd} and we refer to \cite{FV09} for a survey of the theory of twisted Alexander polynomials.

We say that $\Delta_{N,\phi}^\a\in \zt$ is \emph{monic} if its top coefficient equals $\pm 1$. Note that $\Delta_{N,\phi}^\a$ is palindromic, in particular if its top coefficient equals $\pm 1$, then its bottom coefficient also equals $\pm 1$.
Given a polynomial $p(t)\in \zt$ with $p=\sum_{i=k}^l a_it^i, a_k\ne 0, a_l\ne 0$ we define $\deg(p)=l-k$.

We have the following theorem, that gives a characterization of fibered 3-manifolds in terms of twisted Alexander polynomials.

\begin{theorem}\label{thm:fv08} \label{mainthm} Let $(N,\phi)$ be a manifold pair.
 Then $(N,\phi)$ is fibered if and only if
 for
 any epimorphism $\a:\pi_1(N)\to G$ onto a finite group
the twisted Alexander polynomial $\Delta_{N,\phi}^{\a}\in \zt$ is monic
and
\[ \deg(\Delta_{N,\phi}^{\a})= |G| \, \|\phi\|_{T} + (1+b_3(N)) \div \, \phi_{\a},\]
where $\phi_\a$ denotes the restriction of $\phi:\pi_1(N)\to \Z$ to $\ker(\a)$, and where
we denote by  $\div \,\phi_{\a} \in \N$ the divisibility of $\phi_{\a}$, i.e.
\[ \div \, \phi_{\a} =\max\{ n\in \N \, |\, \phi_{\a}=n\psi \mbox{ for some }\psi:\ker(\a)\to \Z\}.\]
 \end{theorem}

The `only if' direction has been shown at various levels of generality by
Cha \cite{Ch03}, Kitano and Morifuji \cite{KM05}, Goda, Kitano and Morifuji \cite{GKM05}, Pajitnov \cite{Pa07}, Kitayama \cite{Kiy08}, \cite{FK06} and \cite[Theorem~6.2]{FV09}.
We will also outline the proof in Section \ref{section:twisted}.
The `if' direction is the main result of \cite{FV08c}.

The main goal of this paper is to provide a summary of the proof, and to show a few ways that the approach can be generalized to prove new results.

Revisiting the proof of Theorem \ref{thm:fv08} will also show that
in fact the following refinement of Theorem \ref{thm:fv08} holds:

\begin{theorem} \label{thm:fv08solv} Let $(N,\phi)$ be a manifold pair such that $\pi_1(N)$ is residually finite solvable (we refer to Section \ref{section:step5} for the definition).
 Then $(N,\phi)$ is fibered if and only if
 for
 any epimorphism $\a:\pi_1(N)\to G$ onto a finite solvable group
the twisted Alexander polynomial $\Delta_{N,\phi}^{\a}\in \zt$ is monic
and if the following equality holds
\[ \deg(\Delta_{N,\phi}^{\a})= |G| \, \|\phi\|_{T} + (1+b_3(N)) \div \, \phi_{\a}.\]
 \end{theorem}

Note that 3-manifolds with  residually finite solvable fundamental group are fairly frequent.
For example, as we will see in Theorem \ref{thm:af09}, any 3-manifold has a finite cover such that its fundamental group is residually finite solvable.
Also note that  the fundamental group of a fibered 3-manifold is always residually finite solvable.


%

\subsection{Fibered manifolds and symplectic 4--manifolds}

The main application of the ``if" direction of Theorem \ref{thm:fv08} is in the proof of the following Theorem.

\begin{theorem} \label{thm:tc}
Let $N$ be a closed 3--manifold. Then $S^1\times N$ is symplectic if and only if $N$ is fibered.
\end{theorem}

\begin{proof} \textit{(outline)}
We first consider the `if' direction. This direction was first proved by Thurston \cite{Th76}, and we present here a proof that is well-known to the experts.
Let $p:N\to S^1$ be a fibration. We write $\psi=p^*(dt)$ where $dt$ is the canonical non-degenerate closed 1--form  on $S^1=\R/\Z$.
By \cite{Ca69} we can find a metric on $N$ such that $\psi$ is harmonic. Denote by $*\psi$ the dual closed 2--form.
We now consider the 1--form $ds$ on $S^1$ as a 1--form on $S^1\times N$ by the pull back operation,
similarly we consider $\psi$ and $*\psi$ as forms on $S^1\times N$.
With this convention we now define:
\[ \omega=ds \wedge \psi +*\psi.\]
Clearly $\omega$ is a closed 2--form on $S^1\times N$. We furthermore calculate that
\[ \omega \wedge \omega = 2  ds \wedge \psi \wedge *\psi.\]
But since $\psi$ is non-zero everywhere, it follows that $\psi \wedge *\psi$ is a 3-form on $N$ which is non-zero everywhere.
Hence $\omega\wedge \omega$ is a 4-form on $S^1\times N$ which is non-zero everywhere. This shows that $\omega$ is a symplectic form on $S^1\times N$.

The `only if' direction  follows by showing that if $S^1 \times N$ is symplectic, then there exists a class $\phi \in H^1(N,\Z)$ determined by the symplectic form that satisfies the constraints described in Theorem \ref{mainthm}. This is achieved building on an idea of Kronheimer in \cite{Kr99}.
Suppose that $N$ is a closed 3-manifold such that $S^1\times N$ admits a symplectic form $\omega$. Without loss of generality we can assume that $\omega$ represents an integral class. Taubes proved in \cite{Ta94,Ta95} that the Seiberg--Witten invariants of a symplectic $4$--manifold satisfy very stringent constraints, that can be viewed as akin to a condition of ``monicness". This, together with the relation between Seiberg--Witten invariants of $S^1 \times N$ and the Alexander polynomial of $N$, due to Meng and Taubes, translates in the condition that $\Delta_{N,\phi}$ is monic,
where $\phi \in H^1(N,\Z)$ is the K\"unneth component of $[\omega]\in H^2(S^1 \times N;\Z)\cong H^1(N;\Z)\oplus H^2(N;\Z)$.
 Next, thanks to a theorem of Donaldson (\cite{Do96}), there exists a symplectic surface dual to (a sufficiently large multiple of) the symplectic form. Such surface satisfies the usual adjunction formula for symplectic surfaces. This formula, played against Kronheimer's adjunction inequality for manifolds of type $S^1 \times N$, gives a constraint on the top degree of the Alexander polynomial $\Delta_{N,\phi}$ in terms of the Thurston norm, more precisely \[ \deg(\Delta_{N,\phi})= \, \|\phi\|_{T} + 2 \div \, \phi.\]  The constraints above hold for all finite covers of $N$, as all finite covers of $S^1 \times N$ are symplectic as well. The connection between the twisted Alexander polynomials of $N$ and the ordinary Alexander polynomials of the finite covers of $N$ entails at this point that for any epimorphism $\a:\pi_1(N)\to G$ to a finite group the twisted Alexander polynomial $\Delta_{N,\phi}^{\a}$ is monic and \[ \deg(\Delta_{N,\phi}^{\a})= |G| \, \|\phi\|_{T} + 2 \div \, \phi_{\a}.\] (We refer to \cite{FV08a} for the details of the argument). The theorem follows at this point from Theorem \ref{mainthm}.
\end{proof}

We would like to mention that an alternative proof of the ``only if" direction of Theorem \ref{thm:tc} in the case that $b_1(N)=1$, and under a technical condition in the general case, follows from combining the work of Kutluhan--Taubes \cite{KT09}, Kronheimer--Mrowka \cite{KM08} and  Ni \cite{Ni08}. This proof requires  a more sophisticated study of the Seiberg--Witten theory of $S^1 \times N$ in the symplectic case.

\subsection{Non--fibered manifolds and vanishing twisted Alexander polynomials} \label{nonfib}

It is natural to ask whether the conditions in Theorem \ref{thm:fv08} can be weakened. In particular, in light of some partial results discussed below, we propose the following conjecture.

\begin{conjecture}\label{conj:zeroalex} \label{conj:alexzero}
Let $(N,\phi)$ be a manifold pair. If $(N,\phi)$ is not fibered, then there exists an epimorphism
$\a:\pi_1(N)\to G$ onto a finite group $G$ such that $\Delta_{N,\phi}^\a=0\in \zt$.
\end{conjecture}


Besides the interest \textit{per se} in sharpening the results of Theorem \ref{mainthm} there are other reasons to investigate Conjecture \ref{conj:zeroalex}. First, the proof of Theorem \ref{thm:tc} would be significantly simplified, bypassing the use of Kronheimer's refined adjunction inequality: Taubes' nonvanishing result for Seiberg--Witten invariants of symplectic manifolds would suffice to carry the argument. But more importantly, Conjecture \ref{conj:zeroalex} would imply a result akin to Theorem \ref{thm:tc} for all symplectic $4$--manifolds that carry a free circle action. For those manifolds, in fact, a refined adjunction inequality in the spirit of \cite{Kr99} does not seem available, and Taubes' constraints translate to a mere monicness
of the twisted Alexander polynomials of the orbit space. We state these observations in the following form, referring to \cite{FV07b} for details and for the extent to which the converse holds.

\begin{theorem} Let $M$ be a $4$--manifold which carries a free circle action with orbit space $N$.
If Conjecture \ref{conj:zeroalex} holds for $N$, then $M$ admits a symplectic structure only if $N$ fibers over the circle.
 \end{theorem}

We also refer to \cite{Ba01} and \cite{Bo09} for related work on the problem of determining which 4--manifolds with a free circle action admit a symplectic structure.
We refer also to the work by Silver and Williams \cite{SW09a,SW09b} and by Pajitnov \cite{Pa08}
for  several interesting connections of Conjecture \ref{conj:zeroalex} to other problems in $3$--dimensional topology.

Conjecture \ref{conj:zeroalex} can be proven to hold for various classes of manifolds. In order to describe them in detail, we must introduce some definitions.

Let $\pi$ be a group and $\G\subset \pi$ a subgroup. We say $\G$ is \emph{separable} if for any $g\in \pi\sm \G$ there exists an epimorphism $\a:\pi\to G$ onto a finite group $G$ such that
$\a(g)\not\in \a(\G)$. Put differently, we can tell that $g$ is not in $\G$ by going to a finite quotient.
We say $\pi$ is \emph{locally extended residually finite (LERF)} if any finitely generated subgroup of $\pi$ is separable.

The following theorem proves Conjecture \ref{conj:zeroalex} in various special cases:

\begin{theorem} \label{thm:sepintro}
Let $(N,\phi)$ be a manifold pair. Suppose that $\Delta_{N,\phi}^\a\ne 0\in \zt$ for any epimorphism $\a:\pi_1(N)\to G$ onto a finite group $G$.
Furthermore suppose that one of the following holds:
\bn
\item $N=S^3\sm \nu K$ and $K$ is a genus one knot,
\item $\tnphi=0$,
\item $N$ is a graph manifold,
\item $\pi_1(N)$ is LERF.
\en
Then $(N,\phi)$ fibers over $S^1$.
\end{theorem}

We refer to \cite[Theorem~1.3]{FV07a} and \cite[Theorem~1,~Proposition~4.6,~Corollary~5.6]{FV08b} for details and proofs.
Note that it is conjectured (cf. \cite{Th82}) that $\pi_1(N)$ is LERF for any hyperbolic 3--manifold $N$.

In Section \ref{section:vr} we will give new conditions under which Conjecture \ref{conj:alexzero} holds.
\\

\noindent \textbf{Acknowledgments.}
The first  author would like to express his gratitude to the organizers of the Georgia International Topology Conference 2009 for the opportunity to speak and
for organizing a most enjoyable and interesting meeting.

\section{Twisted invariants of 3--manifolds}
\label{sectionufd}
\label{section:twialex}
\label{section:twisted}

We recall the definition of twisted homology and cohomology and their basic properties.
Let $X$ be a  topological space and let $\rho:\pi_1(X)\to \gl(n,R)$ be a representation.
Denote by $ \widetilde{X}$ the universal cover of $X$.
Letting $\pi=\pi_1(X)$, we use
the representation $\rho$ to regard $R^n$ as a left $\Z[\pi]$--module.
The chain complex $C_*(\widetilde{X})$ is also a left $\Z[\pi]$--module via deck transformations.
Using the natural involution $g\mapsto g^{-1}$ on the group ring $\Z[\pi]$, we can  view $C_*(\widetilde{X})$ as a right $\Z[\pi]$--module and  form
the twisted homology groups
$$ H_*^\rho(X;R^n)=H_*(C_*(\widetilde{X})\otimes_{\Z[\pi]}R^n).$$

For most of the paper we will be interested in a particular type of representation.
Let $\phi\in H^1(X;\Z)$ and let
$\a:\pi_1(X)\to \gl(n,\Z)$ be a representation. We can now define a left $\Z[\pi_1(X)]$--module
structure on $\Z^n\otimes_\Z \zt=:\znt$ as follows:
\[  g\cdot (v\otimes p):= (\a(g)\cdot v)\otimes (\phi(g)\cdot p) = (\a(g) \cdot v)\otimes (t^{\phi(g)}p), \]
where $g\in \pi_1(X), v\otimes p \in \Z^n\otimes_\Z \zt = \znt$.
Put differently, we get a
representation $\a\otimes \phi:\pi_1(X)\to \gl(n,\zt)$.

We call the resulting twisted module $H_1^{\a\otimes \phi}(X;\znt)$ the
 \emph{twisted Alexander module of $(X,\phi,\a)$}.
When $\phi$ and $\a$ are understood, then we just  write $H_*(X;\znt)$.
 Now suppose  $X$ has finitely many cells in all dimensions.
Using that $\zt$ is a Noetherian UFD it follows that$H_i^{\a\otimes \phi}(X;\znt)$ is  a finitely generated module over
 $\zt$. We now denote by $\Delta^{\a}_{X,\phi,i}\in \zt$ the order of
  $H_1^{\a\otimes \phi}(X;\znt)$ and refer to it as the \emph{twisted Alexander polynomial of $(X,\phi,\a)$}.
 We refer to \cite{Tu01} or \cite[Section~2]{FV09}  for the precise definitions.
Note that the twisted Alexander polynomials are well--defined up to multiplication by an element of the form $\pm t^k, k\in \Z$.

We adopt the convention that we drop $\a$ from the notation if $\a$ is the trivial representation to $\gl(1,\Z)$.
If $\a:\pi_1(N)\to G$ is a homomorphism to a finite group $G$, then we get the regular representation $\pi_1(N)\to G\to \aut(\Z[G])$ where the second map is given by left multiplication. We can identify $\aut(\Z[G])$ with $\gl(|G|,\Z)$ and we obtain the corresponding twisted Alexander polynomial $\Delta_{N,\phi}^\a$.

As an example we give an outline of the proof of the `only if' direction in Theorem \ref{mainthm}.

\begin{lemma} \label{lem:alexfib} Let $(N,\phi)$ be a fibered manifold pair.
 Then
 for
 any epimorphism $\a:\pi_1(N)\to G$ onto a finite group
the twisted Alexander polynomial $\Delta_{N,\phi}^{\a}\in \zt$ is monic
and  the following equality holds
\[ \deg(\Delta_{N,\phi}^{\a})= |G| \, \|\phi\|_{T} + (1+b_3(N)) \div \, \phi_{\a}.\]
 \end{lemma}

 \begin{proof}
First note that  exists a short exact Mayer--Vietoris sequence
 \[ 0\to H_1(\S;\Z[G])\otimes \zt \xrightarrow{t\i_+-\i_-} H_1(N\sm \S;\Z[G])\otimes \zt \to H_1(N;\Z[G]\tpm)\to 0,\]
 where $\Sigma$ is a fiber of $(N,\phi)$.
 Note that in the case of a knot complement and untwisted coefficients this is just the usual exact sequence relating the homology of a Seifert surface to the Alexander module of a knot
(cf. e.g. \cite[Theorem~6.5]{Lic97}).
 We write $r=\mbox{rank} \, H_1(\S;\Z[G])$, where the rank is taken as a $\Z$--module.
 Picking a basis (over $\Z$) for $H_1(\S;\Z[G])$ and $H_1(N\sm \S;\Z[G])=H_1(\S\times [0,1];\Z[G])$
 we can represent $\i_\pm$ by $r\times r$-matrices $A_\pm$. It follows from the definition of the Alexander polynomial
 that
 \[ \Delta_{N,\phi}^\a=\det(tA_+-A_-)=t^r\det(A_+)+ \dots+\det(-A_-).\]
 Since $\i_\pm$ are homotopy equivalences it follows that $\det(A_\pm)=\pm 1$. In particular $\Delta_{N,\phi}^\a$ is monic and of degree $r$. It remains to determine $r$. First note that we have
 \[ \sum_{i=0}^2 (-1)^i \mbox{rank} \, H_i(\S;\Z[G])=|G|\cdot \sum_{i=0}^2 (-1)^i \mbox{rank} \, H_i(\S;\Z)
 =|G|\cdot (-\chi(\S))=|G|\, ||\phi||_T.\]
 Furthermore recall that $\S$ is closed if and only if $N$ is closed. The formula for $r$ now follows from a direct calculation of the rank of $H_0(\S;\Z[G])$ and from duality in the case that $\S$ is closed.
 \end{proof}

\section{Summary of the proof of Theorem \ref{mainthm}}

Our goal is to give an outline of the proof of Theorem \ref{mainthm} given in \cite{FV08c}
without descending into the many technical details required in a rigorous write up.
We are acutely aware of the fact that the proof in \cite{FV08c} hides the forest behind a wall of trees.

\subsection{Step A: First observations}\label{section:step1}

Let $(N,\phi)$ be a manifold pair and $k\in \N$. Note that $(N,\phi)$ fibers if and only if $(N,k\phi)$ fibers and note that $||k\phi||_T=k||\phi||_T$. It follows now easily that it suffices to prove Theorem \ref{thm:fv08} for primitive $\phi\in H^1(N;\Z)$.\\

\noindent \textbf{Theorem A.}\emph{
Let $(N,\phi)$ be a manifold pair with $\phi\in H^1(N;\Z)$ primitive.
Assume that  $\Delta_{N,\phi}\ne 0$. Then the following hold:
\bn
\item
There
exists a connected Thurston norm minimizing surface
$\S$ dual to $\phi$.
\item Any connected  surface  $\S$ dual to $\phi$ intersects any boundary torus, in particular
$\S$ is closed if and only if $N$ is closed.
\item If $\Delta_{N,\phi}^\a\ne 0$ for any epimorphism $\a:\pi_1(N)\to G$ onto a finite group, then $N$ is irreducible.
\item The pair $(N,\phi)$ fibers over $S^1$ if and only if the maps
$\i_\pm:\pi_1(\S)\to \pi_1(N\sm \nu \S)$ are isomorphisms.
\en}

\begin{proof}
 If $\Delta_{N,\phi}\ne 0$, then it follows from  \cite[Section~4~and~Proposition~6.1]{McM02} that  there
exists a connected Thurston norm minimizing surface
$\S$ dual to $\phi$.

Now let $\S$ be any connected surface dual to $\phi$. Suppose that there exists a boundary torus $T$ of $N$ which $\S$ does not intersect. Then $T$ lifts to the infinite cyclic cover $\hat{N}$ of $N$ determined by $\phi : \pi_1(N) \to \Z$,
in particular $\hat{N}$ contains infinitely many tori in its boundary. A standard argument now shows that $b_1(\hat{N})=\infty$, but it is well-known (cf. \cite{Tu01}) that $b_1(\hat{N})=\deg \Delta_{N,\phi}$.

Statement (3) follows from an argument of McCarthy \cite{McC01} (see also \cite[Lemma~7.1]{FV08c} and \cite{Bo09}).
Note that the proof of (3) relies on the fact that 3-manifold groups are residually finite, which is a consequence of the proof of the Geometrization Conjecture (cf. \cite{Th82} and \cite{He87}).
The final statement is a consequence of Stallings' fibering theorem (\cite{St62} and \cite{He76}).
\end{proof}

Throughout this section  $\S$ will always denote a connected Thurston norm minimizing surface dual to $\phi$.
We write $M=N\sm \nu \S$ and denote the two canonical inclusion maps of $\S$ into $\partial M$ by $\i_\pm$.  Since $\S\subset N$ is Thurston norm minimizing  it follows from Dehn's lemma that the inclusion induced maps
$\pi_1(\S)\to \pi_1(N)$ and $\pi_1(M)\to \pi_1(N)$ are injective. In particular we can view $\pi_1(\S)$ and $\pi_1(M)$ as subgroups of $\pi_1(N)$.

Given an epimorphism  $\a:\pi_1(N)\to G$ onto a finite group $G$ we say $\Delta_{N,\phi}^\a$
 has \emph{Property (M)} if
 $\Delta_{N,\phi}^{\a}\in \zt$ is monic
and if \[ \deg(\Delta_{N,\phi}^{\a})= |G| \, \|\phi\|_{T} + (1+b_3(N)) \div  \phi_{\a}\] holds.

\subsection{Step B: Extracting information from twisted Alexander polynomials}\label{section:step2}

In view of Theorem A our strategy is now to
  translate the information coming from twisted Alexander polynomials into information on the maps $\i_\pm:\pi_1(\S)\to \pi_1(M)$.
We start with considering the untwisted polynomial:

\begin{lemma} \label{lem:untwistedh1}
If $\Delta_{N,\phi}$ has Property (M), then the maps $\i_\pm:H_1(\S;\Z)\to H_1(M;\Z)$ are isomorphisms.
\end{lemma}

This lemma is well-known in the case of the untwisted Alexander polynomial for knots. An early reference is given by \cite{CT63}
but see also \cite[Proposition~3.1]{Ni07} or \cite{GS08}.

\begin{proof}
We translate the information on twisted Alexander polynomials into information on the maps $\i_\pm:\pi_1(\S)\to \pi_1(M)$  by considering, as in Lemma \ref{lem:alexfib}, the following long exact Mayer--Vietoris sequence:
\be \label{equ:mv} \ba{cccccccccccccc} \hspace{-0.1cm}&\hspace{-0.1cm}&\hspace{-0.1cm}&\hspace{-0.1cm}\dots \hspace{-0.1cm}&\hspace{-0.1cm}\to\hspace{-0.1cm}&\hspace{-0.1cm} H_2(N;\Z\tpm)
\hspace{-0.1cm}&\hspace{-0.1cm}\to\hspace{-0.1cm}&\hspace{-0.1cm}\\[0.1cm]
 \to \hspace{-0.1cm}&\hspace{-0.1cm}H_1(\S;\Z)\otimes \zt \hspace{-0.1cm}&\hspace{-0.1cm}\xrightarrow{t\i_+-\i_-}
\hspace{-0.1cm}&\hspace{-0.1cm}H_1(M;\Z)\otimes \zt \hspace{-0.1cm}&\hspace{-0.1cm}\to\hspace{-0.1cm}&\hspace{-0.1cm}
H_1(N;\Z\tpm)\hspace{-0.1cm}&\hspace{-0.1cm}\to\hspace{-0.1cm}&\hspace{-0.1cm}\\[0.1cm]
\to \hspace{-0.1cm}&\hspace{-0.1cm}H_0(\S;\Z)\otimes \zt \hspace{-0.1cm}&\hspace{-0.1cm}\xrightarrow{t\i_+-\i_-}
\hspace{-0.1cm}&\hspace{-0.1cm}H_0(M;\Z)\otimes \zt \hspace{-0.1cm}&\hspace{-0.1cm}\to\hspace{-0.1cm}&\hspace{-0.1cm}
H_0(N;\Z\tpm)\hspace{-0.1cm}&\hspace{-0.1cm}\to\hspace{-0.1cm}&\hspace{-0.1cm} 0.
\ea \ee
 Now suppose that
 $\Delta_{N,\phi}\in \zt$ is monic
and that\[ \deg \Delta_{N,\phi}= \|\phi\|_{T} + (1+b_3(N))\] holds.
Note that this in particular implies that $H_1(N;\zt)$ is $\zt$-torsion, since $H_0(\S;\Z)\otimes \zt$ is a free $\zt$-module it follows immediately that the map  $H_1(N;\zt)\to H_0(\S;\Z)\otimes \zt$  is the trivial map.
Now recall that $\S$ is connected and that $\S$ is closed if and only if $N$ is closed. In our context this implies that $ \|\phi\|_{T} + (1+b_3(N))$ equals twice the genus $g$ of $\S$.
Using elementary arguments one can show that $H_1(M;\Z)$ is a free abelian group of the same rank as $H_1(\S;\Z)$, namely $2g$.
Picking bases for $H_1(\S;\Z)$ and $H_1(M;\Z)$ we denote the corresponding $2g\times 2g$--matrices for $\i_\pm$ by $A_\pm$.
It now follows from the definition of the Alexander polynomial
that
\[ \Delta_{N,\phi}=\det(tA_+-A_-)= \det(A_+)t^{2g}+\dots+\det(-A_-).\]
Recall that we assumed that $\Delta_{N,\phi}$ is monic and that $\deg \Delta_{N,\phi}=2g$.
Since $\Delta_{N,\phi}$ is palindromic it now follows that $A_-$ and $A_+$ are invertible matrices, in particular the maps $\i_\pm:H_1(\S;\Z)\to H_1(M;\Z)$ are isomorphisms.
\end{proof}

Clearly the conclusion of the claim is not enough to deduce that $\pi_1(\S)\to \pi_1(M)$ is an isomorphism. In fact there exist many non-fibered knots whose  Alexander polynomial has Property (M).
We will therefore use the information coming from all twisted Alexander polynomials.

Using the idea of the proof the previous claim we can show the following (we refer to \cite[Theorem~3.2]{FV08c} for details).
If $\a:\pi_1(N)\to G$ is an epimorphism onto a finite group such that  $\Delta_{N,\phi}^\a$ has Property (M),
then the maps $\i_\pm:H_1(\S;\Z[G])\to H_1(M;\Z[G])$ are isomorphisms. In fact, considering the `$H_0$-part' of the Mayer--Vietoris sequence (\ref{equ:mv}) we see that the assumption $\Delta_{N,\phi}^\a\ne 0$ implies  that
the maps $ \i_\pm:H_0(\S;\Z[G])\to H_0(M;\Z[G])$
are isomorphisms. Using well-known properties of  0-th homology groups (cf. e.g. \cite[Section~VI]{HS97}) this condition is equivalent to
\[ \im\{\pi_1(\S)\xrightarrow{\i_\pm} \pi_1(M)\xrightarrow{\a} G\}=\im\{ \pi_1(M)\xrightarrow{\a} G\}.\]
For future reference we now summarize the results of the above discussion in the following theorem.
\\

\noindent \textbf{Theorem B.} \emph{
Let $\a:\pi_1(N)\to G$ be an epimorphism onto a finite group such that
  $\Delta_{N,\phi}^\a\ne 0$, then
\be \label{equ:h0} \im\{\pi_1(\S)\xrightarrow{\i_\pm} \pi_1(M)\xrightarrow{\a} G\}=\im\{ \pi_1(M)\xrightarrow{\a} G\}.\ee
If furthermore
  $\Delta_{N,\phi}^\a$ has Property (M), then
\be \label{equ:h1} \i_\pm:H_1(\S;\Z[G])\to H_1(M;\Z[G])\ee are isomorphisms
}

Our goal now is to show that the information we just obtained from twisted Alexander polynomials is in fact enough to deduce that $\i_\pm: \pi_1(\S)\to \pi_1(M)$ are isomorphisms.

\subsection{Step C: Finite solvable quotients}\label{section:step4}


First recall that in the untwisted case we obtained the following conclusion:  if $\Delta_{N,\phi}$ has Property (M), then
the maps $\i_\pm:H_1(\S;\Z)\to H_1(M;\Z)$ are isomorphisms. Another way of saying this is that the maps $\i_\pm:\pi_1(\S)\to \pi_1(M)$ `look like an isomorphism on the abelian level'.
Our goal is now to show that if all twisted Alexander polynomials corresponding to finite solvable groups have Property (M), then the maps $\i_\pm:
\pi_1(\S)\to \pi_1(M)$ `look like isomorphisms on the finite solvable level'.
More precisely, we will prove the following theorem.
\\

\noindent \textbf{Theorem C.}\emph{
Let $(N,\phi)$ be a manifold pair such that for any epimorphism $\pi_1(N)\to S$ onto a finite solvable group the polynomial $\Delta_{N,\phi}^\a$ has Property (M).
Then for any finite solvable group $S$ the induced maps
\[ \i_\pm^*:\hom(\pi_1(M),S)\to \hom(\pi_1(\S),S)\]
are bijections.
}\\

The outline of the proof of Theorem C will require the remainder of this section.
We will now need to introduce a couple of definitions.
Given a solvable group $S$
we denote by $\ell(S)$ its derived length,
i.e. the length of the shortest decomposition into abelian groups.
Note that $\ell(S)=0$ if and only if $S=\{e\}$.

Given $n\in \N\cup \{0\}$ we denote by $\statefs(n)$ the statement that for any finite solvable group $S$
with $\ell(S)\leq n$ the maps
\[ \i^*_\pm:\hom(\pi_1(M),S)\to \hom(\pi_1(\S),S)\]
are bijections.  It is a straightforward exercise to see that $\i_\pm:H_1(\S;\Z)\to H_1(M;\Z)$ are isomorphisms if and only if $\statefs(1)$ holds.


Note that Theorem C says  that $\statefs(n)$ holds for all $n$ if
all twisted Alexander polynomials corresponding to finite solvable groups have Property (M). We will show that this
does indeed hold by induction on $n$.
For the induction argument we use the following  auxiliary statement:
Given $n\in \N\cup \{0\}$ we denote by $\stateh(n)$ the statement that for any epimorphism $\b:\pi_1(M)\to T$ where $T$ is finite solvable with $\ell(T)\leq n$
the maps \[ \i_{\pm} :H_1(\S;\Z[T])\to H_1(M;\Z[T]) \]
are isomorphisms.

\begin{proposition}\cite[Proposition~3.3]{FV08c} \label{prop:ind1}
If $\stateh(n)$ and $\statefs(n)$ hold, then $\statefs(n+1)$ holds as well.
\end{proposition}

\begin{proof}
Let $G$ be a group and $\a:G\to S$ an epimorphism onto a solvable group of derived length $n$. Then we obtain the following short exact sequence
\[ 0\to H_1(G;\Z[S])\to G/[\ker(\a),\ker(\a)]\xrightarrow{\a} S\to 1.\]
In particular if we can control solvable quotients of derived length at most $n$ and the corresponding first homology groups, then we can control solvable information on $G$ up to length $n+1$.
The proposition now follows from elaborating this principle.
We refer to \cite[Section~3.3]{FV08c} for the full details.
\end{proof}

\begin{proposition}\cite[Proposition~3.4]{FV08c} \label{prop:ind2}
Assume that $\Delta_{N,\phi}^\a$ has Property (M) for any epimorphism $\a:\pi_1(N)\to S$ onto a finite solvable group $S$ with $\ell(S)\leq n+1$.
If $\statefs(n)$ holds, then $\stateh(n)$ holds as well.
\end{proposition}

\begin{proof}
In this proof we  find it convenient to introduce the notation
 $\pi=\pi_1(N)$, $A=\pi_1(\S)$ and $B=\pi_1(M)$. Recall that we can view $A$ and $B$ as subgroups of $\pi$.
In fact we can view $\pi$ as an HNN extension of $B$ by $A$, more precisely we have a canonical isomorphism
\[ \pi=\ll B,t \, |\, t\i_-(g)t^{-1}=\i_+(g), g\in A\rr.\]
As usual we will normally just write $\pi=\ll B,t \, |, ti_-(A)t^{-1}=\i_+(A)\rr$.
Let  $\b:B\to T$ be an epimorphism where $T$ is finite solvable with $\ell(T)\leq n$.
Given a finitely generated group $C$ we now define
\[ C(T)=\bigcap\limits_{\g\in Hom(C,T)}\ker(\g).\]
It is straightforward to see that $C/C(T)$ is a finite solvable group with $\ell(C/C(T))\leq n$ (see \cite[Lemma~3.6]{FV08c}).
It is a consequence of $\statefs(n)$ that the homomorphisms
\be \label{equ:aatbbt} \i_\pm: A/A(T) \to B/B(T) \ee
are in fact isomorphisms (see \cite[Lemma~3.6]{FV08c}). In particular we can define an epimorphism
\[ \pi=\ll B,t \, |, t\i_-(A)t^{-1}=\i_+(A)\rr \to   \ll B/B(T),t \, |\,  t\i_-(A/A(T))t^{-1}=\i_+(A/A(T))\rr.\]
It is a consequence of (\ref{equ:aatbbt}) that the above group is in fact a semidirect product, i.e. we have
an isomorphism
\[ \ll B/B(T),t \, |\,  t\i_-(A/A(T))t^{-1}=\i_+(A/A(T))\rr \cong \Z\ltimes B/B(T),\]
where $1\in \Z$ acts on $B/B(T)$ via $\i_-\circ \i_+^{-1}$. Since $B/B(T)$ is finite this automorphism
has finite order, say $k$, therefore there exists an epimorphism
\[ \a:\pi\to \Z\ltimes B/B(T)\to \Z/k\ltimes B/B(T)=:S.\]
Note that $S$ is a finite solvable group of length $n+1$.
It now follows from our assumption that the twisted Alexander polynomial $\Delta_{N,\phi}^\a$ has Property (M).
The information coming from Theorem B is not quite what we wanted, since we replaced $\b:B\to T$ by $\a:B\to S$.
But since $\ker(\a)\subset \ker(\b)$, the latter homomorphism contains in fact the information coming from $\b$,
using a few technical arguments we can now deduce
that
 \[ \i_{\pm} :H_1(A;\Z[T])\to H_1(B;\Z[T]) \]
 is an isomorphism as well.
We refer to \cite[Section~3.4]{FV08c} for the full details.
\end{proof}

Theorem C is now an immediate consequence of Propositions \ref{prop:ind1} and \ref{prop:ind2} and of the fact, observed above, that $\Delta_{N,\phi}$ having Property (M) implies that $\statefs(1)$ holds.
\\

\subsection{Step D: Residually finite solvable fundamental groups}\label{section:step5}
The conclusion of Theorem C loosely says that the maps $\i_\pm$ `look like isomorphisms on the finite solvable level'
if all $\Delta_{N,\phi}^\a$ have Property (M).
With the methods from the previous section  Theorem C is the maximum information on the map $\i_\pm:\pi_1(\S)\to \pi_1(M)$ we can obtain from twisted Alexander polynomials.

In order to analyze the content  of the conclusion of Theorem C we need the following definition.
Let $\PP$ be a property of groups (e.g. finite, finite solvable), then we say that a group $\pi$ is \emph{residually $\PP$} if for any non-trivial $g\in \PP$ there exists a homomorphism $\a:\pi\to G$ to a
group $G$ with Property $\PP$ such that $\a(g)$ is non-trivial. For example it is well-known that surface groups are residually finite solvable, and that 3-manifold groups
are residually finite (cf. \cite{Th82} and \cite{He87}).

On the other hand 3-manifold groups are in general not residually finite solvable. For example if $K$ is a non-trivial knot with Alexander polynomial equal to one,
then standard arguments show that any homomorphism $\pi_1(S^3\sm \nu K)\to S$ to a solvable group $S$ necessarily factors through the abelianization
$\pi_1(S^3\sm \nu K)\to \Z$. In particular $\pi_1(S^3\sm \nu K)$ is not residually finite solvable.
One can use such a knot to construct a manifold pair $(N,\phi)$
where $\i_\pm:\hom(\pi_1(\S),S)\to \hom(\pi_1(M),S)$ is a bijection for any solvable $S$,
but such that $\pi_1(M)$ is not residually solvable. In particular $M$ is not a product.

%

This discussion shows that the conclusion of Theorem C is not strong enough to ensure that $\i_\pm:\pi_1(\S)\to \pi_1(M)$ are isomorphisms,
the problem being that 3-manifold groups are in general not residually finite solvable.

Before we continue we need to introduce a few more notions.
We say a group has a property \emph{virtually}, if there exists a  finite index subgroup  which has this property.
Also recall, that given a prime $p$ a \emph{$p$--group} is a group whose order is  a power of $p$.
If $G$ is a group which is residually a $p$--group, then we will normally just say \emph{$G$ is residually $p$}.

We can now formulate the following recent theorem of Matthias Aschenbrenner and the first author.

\begin{theorem}\cite{AF10}\label{thm:af09}
Let $N$ be a 3-manifold. Then for almost all primes $p$ the group $\pi_1(N)$ is virtually residually $p$.
\end{theorem}

Recall that $p$-groups are finite solvable, in particular Theorem \ref{thm:af09} says that 3-manifold groups
are virtually residually finite solvable.

If $N$ is hyperbolic then the theorem is a consequence of the fact that linear groups are virtually residually $p$ (cf. e.g. \cite[Theorem~4.7]{We73}).
The proof of that fact is so short and elegant that we think it is worthwhile mentioning.

\begin{proof}[Proof of Theorem \ref{thm:af09} for hyperbolic $N$]
We write $\pi=\pi_1(N)$. Since we assume that $N$ is hyperbolic we can assume  that $\pi$ is  a subgroup of $SL(2,\C)$.
Since $\pi$ is finitely generated there exists a finitely generated subring $R$ of $\C$ such that $\pi\subset \gl(n,R)$.
It is well--known that for almost all primes $p$ there exists a maximal ideal $\mm$ of $R$
with $\mbox{char}(R/\mm)=p$ (see  \cite[p.~376f]{LS03}).

Now let $p$ be a prime for which there exists a  maximal ideal $\mm$ of $R$ with $\mbox{char}(R/\mm)=p$.
We will show that $\pi$ is virtually residually $p$.
Before we continue note that  $R/\mm^k$ is a finite ring for any $k\geq 1$
and that  $\bigcap_{k=1}^\infty \mathfrak{m}^k=\{0\}$ by the Krull Intersection Theorem.
For $k\geq 1$ we let
\[ \pi_{k} = \ker\left( \pi\to \gl(n,R)\to \gl(n,R/\mm^k)\right).\]
Each $\pi_k$ is a normal subgroup of $\pi$, of finite index, and clearly $\pi_{k+1}\subset \pi_{k}$ for every $k\geq 1$. Moreover $\bigcap_{k=1}^\infty \pi_{k}=\{1\}$ since $\bigcap_{k=1}^\infty \mathfrak{m}^k=\{0\}$.

We claim that $\pi_1$ is residually $p$. We will prove this by showing that
$\pi_1/\pi_k$ is a $p$-group for any $k$. This in turn follows from showing that
any non--trivial element in $\pi_k/\pi_{k+1}$ has order $p$.
In order to show this pick  $A \in \pi_{k}$.
By definition we can write $$A=\id +C\qquad\text{for some $n\times n$-matrix $C$ with entries in $\mm^k$.}$$
From $p\in \mm$ and $k\geq 1$ we get that
\[ \ba{rcl}  A^p=(\id+ C)^p &=&\id+pC+\frac{p(p-1)}{2}C^2+\dots+C^p\\
&=& \id+\text{(some $n\times n$-matrix with entries in $\mm^{k+1}$).}\ea \]
 Hence $A^p\in \pi_{k+1}$.
%
\end{proof}

The combination of Theorem C and \ref{thm:af09} shows that proving Theorem \ref{thm:fv08}
becomes much easier, if  we can go to finite covers. Fortunately the  following lemma tells us that we can indeed do so:

\begin{lemma}\label{lem:finitecover}
Let $p:N'\to N$ be a finite cover and let $\phi'=p^*(\phi)$. Then the following hold:
 \bn
 \item $(N,\phi)$ fibers if and only if $(N',\phi')$ fibers,
\item if $\Delta_{N,\phi}^\a$ has Property (M) for any epimorphism $\a$ from $\pi_1(N)$ onto a finite group, then
$\Delta_{N,\phi}^\a$ has Property (M) for any epimorphism $\a$ from $\pi_1(N')$ onto a finite group.
\en
\end{lemma}

\begin{proof}
 The first statement can for example be proved  using Stallings' fibering theorem \cite{St62}:
Indeed, $(N,\phi)$ fibers if and only if $\ker(\phi)$ is finitely generated and
$(N',\phi')$ fibers if and only if $\ker(\phi')$ is finitely generated. But $\ker(\phi')$ is subgroup of $\ker(\phi)$ of finite index. In particular if one is finitely generated, then so is the other.

The second statement is fundamentally just an application of Shapiro's lemma, which says that the homology of a finite cover of a space $N$ is nothing but the twisted homology of $N$. Making this principle work in this context is a little delicate though, and we refer the reader to \cite[Lemma~7.6]{FV08c} for the details.
\end{proof}

The following is now an immediate corollary to Theorem \ref{thm:af09} and Lemma \ref{lem:finitecover}.
\\

\noindent \textbf{Theorem D.}
\emph{Suppose the conclusion of Theorem \ref{thm:fv08} holds for all 3-manifolds such that $\pi_1(N)$ is residually finite solvable,
then Theorem \ref{thm:fv08} holds for all 3-manifolds.
}\\

Note that the original proof of Theorem \ref{thm:fv08} (that appeared before the first author and M. Aschenbrenner completed the proof of Theorem \ref{thm:af09}) required in \cite[Section~6]{FV08c} a rather convoluted argument based on the study of residual properties of each piece of the JSJ decomposition of $N$. The result of Theorem \ref{thm:af09} therefore greatly simplifies the argument.


\subsection{Step E. Reformulation in terms of sutured manifolds and Agol's theorem}\label{section:step6}

In our final step we find it convenient to switch to the language of sutured manifolds.
A \emph{sutured manifold} is a triple $(M,\S_-,\S_+)$ where $M$ is an oriented 3-manifold, $\S_\pm$ are (possibly disconnected) disjoint oriented subsurfaces of $\partial M$  with the following properties:
\bn
\item the orientation of $\S_+$ agrees with  the orientation of $\partial M$,
\item the orientation of $\S_-$ is the opposite orientation of $\partial M$,
\item the closure of $\partial M\sm \S_- \cup \S_+$ consists of a union of annuli $A_1,\dots,A_n$ such that
for any $i$ the boundary of $A_i$ consists of a boundary curve of $\S_-$ and of a boundary curve of $\S_+$.
Furthermore the boundary curves have to be oriented the same way.
\en
A  sutured manifold  $(M,\S_-,\S_+)$ is called \emph{taut} if $M$ is irreducible and if $\S_\pm$  are Thurston norm minimizing in their homology class in $H_2(M,\partial \S_\pm;\Z)$.
We refer to \cite{Ju06}, \cite[Definition~2.6]{Ga83} or \cite[p.~364]{CC03} for more on  sutured manifolds.

The following are the two most important types of examples for us:
\bn
\item If $\S$ is a  oriented surface, then $(\S\times [-1,1], -\S\times -1, \S\times 1)$ is a taut sutured manifold.
We will refer to it as a \emph{product sutured manifold}.
\item Let $N$ be an irreducible 3--manifold with empty or toroidal boundary. Let $\S$ be a Thurston norm minimizing surface  which intersects all boundary tori of $N$.
Denote by $M$ the result of cutting  $N$  along $\S$ and denote by  $\S_\pm$  the two copies of $\S$ in $M$. Then $(M,\S_-,\S_+)$ is a taut  sutured manifold.
\en

\noindent \textbf{Theorem E.}
\emph{
Let $(M,\S_-,\S_+)$ be a taut  sutured manifold.
Suppose that $\pi_1(M)$ is residually finite solvable and suppose that for any finite solvable group $S$ the induced maps
\[ \i^{*}_{\pm}:\hom(\pi_1(M),S)\to \hom(\pi_1(\S_\pm),S)\]
are bijections. Then $M$ is a product on $\S_\pm$.} \\

Theorem \ref{thm:fv08} is an immediate consequence of  Theorems A, C, D and E,
and  Theorem \ref{thm:fv08solv} is an immediate  consequence of Theorems A, C and E.
\\

Note that the statement of Theorem E can be generalized to a question about groups in general: Let $\PP$ be a property
of groups,
let $\varphi:A\to B$ be a homomorphism of finitely presented groups which are residually $\PP$ such that for any group $G$ with Property $\PP$ the map
$\hom(B,G)\to \hom(A,G)$ is a bijection. Does this imply that $\varphi$ is an isomorphism?
For  $\PP=\{\mbox{finite}\}$ this question goes back to Grothendieck \cite{Gr70} and was answered in the negative by Bridson and Grunewald \cite{BG04}.
We refer to  \cite{AHKS07} for more on the case $\PP=\{\mbox{finite  solvable}\}$.

This excursion into group theory shows that in order to prove Theorem E we can not rely on a miracle in group theory,
but we need a miracle which comes from our 3-dimensional setting.
This miracle is provided by a stunning theorem of Agol \cite{Ag08}.
To explain it we need one more definition.

A group $\pi$ is called \emph{residually finite $\Q$--solvable} or \emph{RFRS} if there
exists a filtration  of groups $\pi=\pi_0\supset \pi_1 \supset \pi_2\dots $
such that the following hold:
\bn
\item $\cap_i \pi_i=\{1\}$,
\item   $\pi_i$ is a normal, finite index subgroup of  $\pi$ for any $i$,
\item for any $i$ the map $\pi_i\to \pi_i/\pi_{i+1}$ factors through $\pi_i\to H_1(\pi_i;\Z)$,
\item for any $i$ the map $\pi_i\to \pi_i/\pi_{i+1}$ factors through $\pi_i\to H_1(\pi_i;\Z)/\mbox{torsion}$.
\en
Note that conditions (1), (2) and (3) are equivalent to saying that $\pi$ is residually finite solvable.
But condition (4) means that the RFRS condition is considerably more restrictive.
The notion of an RFRS group was introduced by Agol \cite{Ag08},
we refer to Agol's paper for more information on RFRS groups.
For our context it is important to note that free groups and surface groups are RFRS. Indeed, it is well-known that these groups are residually finite solvable, in particular there exists
a sequence $\pi_i$ with Properties (1), (2) and (3). But the extra condition (4) is now always satisfied since the first homology of any finite index subgroup of a free group or a surface group is always torsion free.

Given a sutured manifold $M=(M,\S_-,\S_+)$ the double $D_M$ is defined to be the double of $M$ along $\S_-$ and $\S_+$.
Note that the annuli $\partial M\sm (\S_-\cup \S_+)$ give rise to toroidal boundary components of $D_M$.

The following theorem is   implicit in the proof of   \cite[Theorem~6.1]{Ag08}.

\begin{theorem}[Agol] \label{thm:agol2}
Let $M=(M,\S_-,\S_+)$ be a connected, taut sutured manifold which is not a product sutured manifold.
Suppose that  $\pi_1(M)$ is RFRS.
Then there exists an epimorphism $\a:\pi_1(M)\to S$ onto a finite solvable group,
such that the corresponding cover $\ti{M}=(\ti{M},\ti{\S}_-,\ti{\S}_+)$ of $M=(M,\S_-,\S_+)$
has the property that the class $[\ti{\S}_-]\in H_2(D_{\ti{M}},\partial D_{\ti{M}};\Z)$
lies on the closure of the cone over a fibered face
of the Thurston norm ball of  $D_{\ti{M}}$.
\end{theorem}

Before we delve into the details of the proof of Theorem E, let us take a step back and think about what Theorem \ref{thm:agol2} does for us.
Agol's theorem has as input information on finite solvable quotients of $\pi_1(M)$ and as output it gives us a strong topological conclusion.
This is exactly the type of statement we want to make in Theorem E. It is now just a matter of time till bending and twisting
turns Theorem \ref{thm:agol2}  into a proof of  Theorem E.

\begin{proof}[Proof of Theorem E]
Let $M=(M,\S_-,\S_+)$ be a taut  sutured manifold such that $\pi_1(M)$ is residually finite solvable and such  that for any finite solvable group $S$ the induced maps
\be \label{equ:samehom} \i^*:\hom(\pi_1(M),S)\to \hom(\pi_1(\S_\pm),S)\ee
are bijections. We have to show that $M$ is a product sutured manifold.

Let us now suppose that $M$ is not a product sutured manifold.  We  write $\S=\S_-$.
Recall that we pointed out above that the surface group $\pi_1(\S)$ is RFRS.
By assumption $\pi_1(M)$ is residually finite solvable  and by (\ref{equ:samehom}) the finite solvable quotients of $\pi_1(M)$ and $\pi_1(\S)$ `look the same'.
It is now fairly elementary to show that $\pi_1(M)$ is also RFRS (see \cite[Section~4.2]{FV08c} for details).

We can thus apply Theorem \ref{thm:agol2} to the taut sutured manifold $(M,\S_-,\S_+)$. In fact applying arguments similar to the ones used in Lemma \ref{lem:finitecover}
we can without loss of generality assume that
already the class $[{\S}_-]\in H_2(D_M,\partial D_M;\Z)$
lies on the closure of the cone over a fibered face $F$
of the Thurston norm ball of  $D_M$. Moreover, as by hypothesis $M$ is not a product, we can assume that $[{\S}_-]$ lies in the cone over the boundary of $F$, as otherwise $(N,\phi)$ would fiber already. We refer to \cite[Lemma~4.3]{FV08c} for details.

Note that $D_M$ has an obvious involution $r$ given by `reflection', i.e. interchanging the two copies of $M$.
Also recall that $(\ref{equ:samehom})$ implies in particular that the inclusion induced maps $H_1(\S_\pm;\Z)\to H_1(M;\Z)$ are isomorphisms.
This means that homologically $M$ looks like a product, and hence homologically $D_M$ looks like $S^1\times \S$.
More precisely, there exists a canonical isomorphism $\Z \cdot t\oplus H_1(\S_-;\Z)\xrightarrow{\cong} H_1(D_M;\Z)$,
where $t$ is an oriented curve with $r(t)=-t$ which intersects each of $\S_-$ and $\S_+$ once.
Note that the map $r:H_1(D_M;\Z)\to H_1(D_M;\Z)$ restricts to the identity on $H_1(\S;\Z)$ and sends $t$ to $-t$.

 Applying duality we now obtain a dual isomorphism
 $H_2(D_M,\partial D_M;\Z)=\Z \cdot [\S]\oplus V$ where $r$ acts as $-\id$ on $V$.
Recall that  $r([\S])=[\S]$ and that $[\S]$  sits on the boundary of the face $F$. It follows that $[\S]$ also sits on the boundary of the face $r(F)$.
Clearly $r(F)$ is also a fibered face, and since $r$ acts as $-\id$ on $V$ we see that $F$ and $r(F)$ are distinct faces.
 Also note that by the convexity of the Thurston norm ball $F$ and $r(F)$ can not sit on the same plane.
Schematically we now have the situation presented in Figure \ref{fig:thurston}.
\begin{figure}[h] \begin{center}
 \includegraphics[scale=0.3]{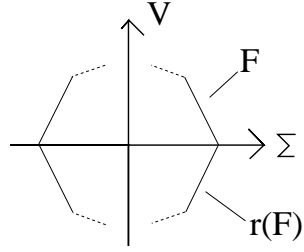} \caption{Thurston norm ball on $H_2(D_M,\partial D_M;\R)=\R\cdot \S\oplus V$.}
\label{fig:thurston}
\end{center}
 \end{figure}

We now consider the information contained in $(\ref{equ:samehom})$ coming from finite metabelian groups.
We see that  $M$ looks like a product on  the `metabelian level', and hence  $D_M$ looks like $S^1\times \S$ on the `metabelian level'.
Now recall that the multivariable Alexander polynomial of a 3-manifold is a metabelian invariant, we conclude that
the multivariable Alexander polynomial of $D_M$ equals the multivariable Alexander polynomial of $S^1\times \S$ which is
well-known to be given by $(1-t)^{-\chi(\S)}$, where $t\in H_1(S^1\times M;\Z)=H_1(D_M;\Z)$ is the same generator introduced above
(we refer to \cite[Lemma~4.9]{FV08c} for details).

The norm ball dual to the Newton polygon of the multivariable Alexander polynomial is called the Alexander norm ball (\cite{McM02}). Note that the Alexander norm ball of $D_M$ is a convex subset of $\hom(H_1(D_M;\Z);\R)=H^2(D_M,\partial D_M;\R)$.
Since the Alexander polynomial is given by $(1-t)^{-\chi(\S)}$ and since $\S$ is dual to $t$ we see that the Alexander norm ball
in our case is given by Figure \ref{fig:alexander}.
\begin{figure}[h] \begin{center}
 \includegraphics[scale=0.3]{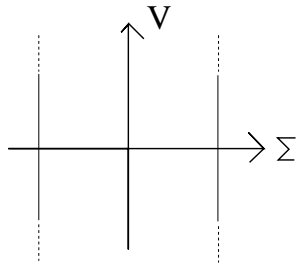} \caption{Alexander norm ball on $H_2(D_M,\partial D_M;\R)=\R\cdot \S\oplus V$.}
\label{fig:alexander}
\end{center}
 \end{figure}
 Note that for fibered classes the Thurston norm and the Alexander norm agree (see \cite{McM02}).
 In particular the distinct fibered faces $F$ and $r(F)$ of the Thurston norm ball have to lie on distinct faces of the Alexander norm ball. But the Alexander norm ball only has two faces, both preserved under reflection.
 This leads to a contradiction, which is schematically indicated by the mismatch of  Figure \ref{fig:thurston} and Figure \ref{fig:alexander}.
 This shows that the assumption that $M$ is not a product leads to a contradiction. We refer to \cite[Section~4]{FV08c} for a formal and completely rigorous version of the above argument.
\end{proof}


\section{Vanishing twisted Alexander polynomials for non--fibered manifolds}\label{section:vr}

Throughout this section we use the notation from the previous section. In particular given a manifold pair $(N,\phi)$ where $\phi \in H^1(N,\Z)$ is a primitive class with $\Delta_{N,\phi}\ne 0$,
we will denote by $\S$ a connected Thurston norm minimizing surface dual to $\phi$ and we will write $M=N\sm \nu \S$. Furthermore we will  denote the two natural inclusion maps of $\S$ into $M$ by $\i_\pm$.
We recall the following theorem, whose proof we had outlined above:\\

\noindent \textbf{Theorem B.} \emph{
Let $(N,\phi)$ be a manifold pair and let $\a:\pi_1(N)\to G$ be an epimorphism onto a finite group such that
  $\Delta_{N,\phi}^\a\ne 0$, then
\[ \im\{\pi_1(\S)\xrightarrow{\i_\pm} \pi_1(M)\xrightarrow{\a} G\}=\im\{ \pi_1(M)\xrightarrow{\a} G\}.\]
}

We will see in this section that Theorem B can be used in many situations to show that a non-fibered manifold pair has zero twisted Alexander polynomials.

\begin{theorem}\cite[Theorem~4.2]{FV08b}\label{thm:sep}
Let $(N,\phi)$ be a non-fibered manifold pair such that $\phi$ is dual to a connected incompressible
  surface $\S$.
If $\pi_1(N)$ is LERF, then there exists an epimorphism $\a:\pi_1(N)\to G$ onto a finite group $G$ such that $\Delta_{N,\phi}^\a=0$.
\end{theorem}

\begin{proof}
We write $\S=\S_-$.
By Theorem A we know that the monomorphism  $\pi_1(\S)\to \pi_1(M)$ is not an isomorphism, in particular the set $\pi_1(M)\setminus \pi_1(\S)$ is nonempty. By the separability of $\pi_1(\Sigma)\subset \pi_1(N)$ we can now find for any $g \in \pi_1(M)\setminus \pi_1(\S)$ an epimorphism
$\a:\pi_1(N)\to G$ onto a finite group $G$ such that $\a(g)\not\in \a(\pi_1(\S))$. In particular we have
 \[ \im\{\pi_1(\S)\xrightarrow{\i} \pi_1(M)\xrightarrow{\a} G\}\subsetneq \im\{ \pi_1(M)\xrightarrow{\a} G\}.\]
 The theorem now follows  from Theorem B.
\end{proof}

It is an important open question whether fundamental groups of hyperbolic $3$--manifolds are LERF
(see \cite[Question~15]{Th82}), and various partial results of separability are known.
In particular, Long and Niblo \cite[Theorem~2]{LN91} showed that the subgroup carried by an embedded torus is separable.
It follows that  the separability condition required in the proof of Theorem \ref{thm:sep} is always satisfied if $\S$ is a torus.
We now easily obtain the following.

\begin{theorem}\cite[Proposition~4.6]{FV08b}\label{prop:torus}
Let $N$ be a closed 3-manifold and $\phi\in H^1(N;\Z)$ a non-trivial class with $||\phi||_T=0$.
Then $(N,\phi)$ is fibered if and only if for any  epimorphism $\a:\pi_1(N)\to G$ onto a finite group $G$ we have $\Delta_{N,\phi}^\a\ne 0$.
\end{theorem}

Expanding on the ideas of Theorem \ref{thm:sep} and \ref{prop:torus} one can then continue to prove Theorem
 \ref{thm:sepintro}. We refer to \cite{FV08b} for details.


Unfortunately  not all $3$-manifold groups are LERF (see \cite{NW01}) and little is known even conjecturally about the separability properties of non-geometric 3-manifold groups.
In the remainder of this section we will therefore give two examples of types of non-fibered manifold pairs  where Theorem \ref{thm:sep} cannot be applied,
but where the weaker assumptions of Theorem B allow us to show that these pairs have twisted Alexander polynomials which are zero.

In order to prove our theorems we recall the following result of Long and Niblo \cite{LN91}:
Let $\Sigma$ be an incompressible subsurface of the boundary of a $3$--manifold $M$. Then $\pi_1(\S)\subset \pi_1(M)$ is separable.
This is often referred to as \textit{peripheral subgroup separability}. We will exploit this result in two cases. The first is the case of the double of the complement of a nonfibered surface $\Sigma \subset N$. The second, perhaps of more conceptual breadth, is an application of `virtual retractibility'.

We start with the first case.
Let $W$ be a $3$--manifold with empty or toroidal boundary and let $\Sigma \subset W$ be an incompressible nonseparating connected properly embedded surface. Consider the manifold with boundary $M = W \setminus \nu \Sigma$. This manifold has two copies $\S_\pm$ sitting in the boundary. Consider the double $D_M$ of $M$ along $\S_-\cup \S_+\subset \partial M$.
 The images $\Sigma_{\pm} \subset D_M$ are nonseparating incompressible surfaces which are homologous in $H_2(D_M,\partial D_M;\Z)$.


\begin{theorem} Let $D_M$ be defined as above, and let $\phi \in H^1(D_M,\Z)$ be the primitive class Poincar\'e dual to $[\Sigma_{\pm}]$. If
$(D_M,\phi)$ is a non-fibered pair, then there exists an epimorphism $\a : \pi_1(D_M) \to G$ onto a finite group $G$ such that  $\Delta_{D_M,\phi}^\a =0$. \end{theorem}

\begin{proof}
Suppose that  $(D_M,\phi)$ is a non--fibered pair.
First note that it is well-known that $\S_+\subset D_M$ is a fiber if and only if $M$ is a product on $\S_+$.

Note that we have a folding map $r : D_M \to M$ that is a retraction.
 In particular the induced map in homotopy $r_{*} : \pi_1(D_M) \to \pi_1(M)$ is an epimorphism, and has as right inverse the inclusion--induced map $i_{*}: \pi_1(M) \to \pi_{1}(D_M)$.
Note that both $r_{*}$ and $i_{*}$ restrict to an isomorphism on the proper subgroups of the domain and image determined by a copy of $\pi_1(\Sigma_{+})$.
Consider now the proper subgroup $\pi_1(\Sigma_{+}) \subset \pi_1(M)$; by peripheral subgroup separability, there is an epimorphism $\beta  : \pi_1(M) \to G$ to a finite group such that $\beta(\pi_1(\Sigma_{+})) \subsetneq \beta(\pi_1(M))$.
The surface $\Sigma_{+} \subset D_M$ is an incompressible surface dual to $\phi$; define $Z := D_M \setminus \nu \Sigma_{+}$. By the usual argument based on the incompressibility of $\Sigma_{+}$, $\pi_1(\Sigma_{+})$ can be viewed as subgroup of $\pi_1(Z)$ and the latter is a subgroup of $\pi_{1}(D_M)$.

The inclusion--induced map $i_{*}: \pi_1(M) \to \pi_{1}(D_M)$ has image in $\pi_1(Z)$. It follows that if we let $\a=\beta \circ r_{*} : \pi_1(D_M) \to G$, then we have
\[ \im\{\pi_1(\S_{+})\xrightarrow{\i} \pi_1(Z)\xrightarrow{\a} G\} \subsetneq \im\{ \pi_1(Z)\xrightarrow{\a} G\}.\]
It now follows from  Theorem B that $\Delta_{D_M,\phi}^\a = 0$.

\end{proof}

The second application of peripheral subgroup separability is in the context of \textit{virtual retractions}.
In \cite{LR08} (see also \cite{LR05}), Darren Long and Alan Reid define and explore the notion of \textit{virtual retraction} of a group to one of its finitely generated subgroup, as well as various related properties. As the authors of
 \cite{LR08} discuss, these notions are closely connected with subgroup separability properties of the group.

 %

 We start by giving the proper definitions, from \cite{LR08}, using a notation that adapts to the case we have in mind.

 \begin{definition} Let $\pi$ be a group and $B$ a subgroup. Then a homomorphism $\theta : B \to G$ \textit{extends over the finite index subgroup} ${\hat \pi} \subset \pi$ if $B \subset {\hat \pi}$ and if there exists a homomorphism $\Theta : {\hat \pi} \to G$ such that $\Theta|_{B} = \theta$. \end{definition}

 \begin{definition} Let $\pi$ be a group and $B$ a subgroup. Then $\pi$ \textit{virtually retracts} onto $B$ if the identity homomorphism $\theta = \id_{B}$ extends over some finite index subgroup of $\pi$. \end{definition}

 Quite clearly, if  $\pi$ virtually retracts onto $B$, \textit{any} homomorphism $\theta  : B \to G$ extends over some finite index subgroup of $\pi$. A less trivial fact, observed in \cite[Theorem~2.1]{LR08}, is that if $\pi$ is LERF and if $B$ is finitely generated, then any homomorphism $\theta$ onto a finite group extends over a finite index subgroup.
 In light of that the following  theorem can be viewed as a generalization of Theorem \ref{thm:sep}.


 \begin{theorem} \label{vr}
 Let $(N,\phi)$ be a non-fibered manifold pair.
 Suppose that there exists a connected Thurston norm minimizing surface $\S$ dual to $\phi$
 such that any homomorphism of $\pi_1(N \setminus \nu \S)$ to a finite group extends to a finite index subgroup of $\pi_1(N)$.
Then there exists an epimorphism $\a:\pi_1(N)\to G$ onto a finite group $G$ such that $\Delta_{N,\phi}^{\a} = 0$.
 \end{theorem}

 \begin{proof}
 We write  $M : = N \setminus \nu \Sigma$ and we write  $\pi = \pi_1(N)$, $A = \pi_1(\Sigma)$ and $B = \pi_1(M)$. The manifold $M$ has, as boundary, two copies of $\Sigma$; the incompressibility of $\Sigma$ entails in particular the existence of inclusion--induced injective morphisms $\i_{\pm} : A \hookrightarrow B \subset \pi$.

 As (a copy of) $\Sigma$ occurs as boundary component of $M$, the image under say $i_{+}$ of $A$ in $B$ (that we will denote by $A$ as well) is separable by peripheral subgroup separability. This means that for any element $\gamma \in B \setminus A$ there exist an epimorphism $\theta : B \to H$ onto some finite group such that $\theta(\gamma) \notin \theta(A)$. Pick such an element $\gamma$. By assumption, $\theta$ extends to an epimorphism $\Theta : {\hat \pi} \to H$ where ${\hat \pi} \subset \pi$ is a finite index subgroup. The kernel $\mbox{ker } \Theta \subset {\hat \pi}$ is a normal finite index subgroup of $\pi$. This subgroup may fail to be normal in $\pi$; define $\Gamma = \cap_{g \in \pi} g {\,} (\mbox{ker } \Theta) {\,} g^{-1}$ to be its normal core in $\pi$, a normal finite index subgroup of both ${\hat \pi}$ and $\pi$. Denote by $G := \pi/\Gamma$, a finite group, and let $\alpha : \pi \to G$ be the quotient map. As ${\hat \pi}/\Gamma$ surjects on ${\hat \pi}/\mbox{ker } \Theta$, it is not difficult to verify that the condition $\Theta(A) \subsetneq \Theta(B)$ entails that $\a(A) \subsetneq \a(B) \subset G$. The theorem is now an immediate consequence of Theorem B. \end{proof}

 Clearly, Theorem \ref{vr} applies when $\pi_1(N)$ virtually retracts to $\pi_1(M)$. Examples where this occurs are $3$--manifolds whose fundamental group embeds into an all right hyperbolic Coxeter subgroup of $\mbox{Isom}(\Bbb{H}^n)$ (see \cite[Theorem~2.6]{LR08}): these groups retract to any finitely generated geometrically finite subgroup, and the only finitely generated geometrically infinite subgroups are virtual fiber groups, hence excluded in the statement of Theorem \ref{vr}.

\end{document}